\documentclass{amsart}

\usepackage{amssymb}
\usepackage{microtype}
\usepackage{fullpage}
\usepackage{hyperref}
\usepackage{enumerate}

\hypersetup{colorlinks=true, citecolor=blue, linkcolor=blue, urlcolor=blue, pdfstartview=FitH, pdfauthor=Miko\l aj Fraczyk and Gergely
Harcos and P\'eter Maga, pdftitle=Counting bounded elements of a number field}

\newcommand{\ZZ}{\mathbb{Z}}
\newcommand{\QQ}{\mathbb{Q}}
\newcommand{\RR}{\mathbb{R}}
\newcommand{\CC}{\mathbb{C}}
\newcommand{\FF}{\mathbb{F}}
\newcommand{\mo}{\mathfrak{o}}
\newcommand{\mn}{\mathfrak{n}}
\newcommand{\mpr}{\mathfrak{p}}
\newcommand{\cB}{\mathcal{B}}
\newcommand{\cC}{\mathcal{C}}
\newcommand{\cD}{\mathcal{D}}
\newcommand{\cT}{\mathcal{T}}

\renewcommand{\geq}{\geqslant}
\renewcommand{\leq}{\leqslant}

\newcommand{\GL}{\mathrm{GL}}

\newcommand{\ov}[1]{\overline{#1}}
\newcommand{\Dt}{\Delta_\mathrm{tame}}

\DeclareMathOperator{\vol}{vol}
\DeclareMathOperator{\Hom}{Hom}
\DeclareMathOperator{\Gal}{Gal}
\DeclareMathOperator{\re}{Re}
\DeclareMathOperator{\im}{Im}
\DeclareMathOperator{\rank}{rank}

\theoremstyle{plain}
\newtheorem{theorem}{Theorem}
\newtheorem{proposition}{Proposition}
\newtheorem{corollary}{Corollary}

\theoremstyle{definition}
\newtheorem*{acknowledgement}{Acknowledgements}
\newtheorem*{question}{Question}

\begin{document}

\author{Miko\l aj Fraczyk}
\author{Gergely Harcos}
\author{P\'eter Maga}

\address{Alfr\'ed R\'enyi Institute of Mathematics, Hungarian Academy of Sciences, POB 127, Budapest H-1364, Hungary}\email{fraczyk@renyi.hu, gharcos@renyi.hu, magapeter@gmail.com}
\address{MTA R\'enyi Int\'ezet Lend\"ulet Groups and Graphs Research Group}\email{fraczyk@renyi.hu}
\address{MTA R\'enyi Int\'ezet Lend\"ulet Automorphic Research Group}\email{gharcos@renyi.hu, magapeter@gmail.com}
\address{Institute for Advanced Study,  Princeton NJ, USA}\email{mikolaj@ias.edu}
\address{Central European University, Nador u. 9, Budapest H-1051, Hungary}\email{harcosg@ceu.edu}

\title{Counting bounded elements of a number field}

\begin{abstract}
We estimate, in a number field, the number of elements and the maximal number of linearly independent elements, with prescribed bounds on their valuations. As a by-product, we obtain new bounds for the successive minima of ideal lattices. Our arguments combine group theory, ramification theory, and the geometry of numbers.
\end{abstract}



\thanks{First author supported by ERC grant CoG-648017 and the MTA R\'enyi Int\'ezet Lend\"ulet Groups and Graphs Research Group.
Second and third author supported by NKFIH (National Research, Development and Innovation Office) grant K~119528 and by the MTA R\'enyi Int\'ezet Lend\"ulet Automorphic Research Group. Third author also supported by the Premium Postdoctoral Fellowship of the Hungarian Academy of Sciences.}

\maketitle

\section{Introduction}

It was a decisive moment in the history of mathematics when Minkowski~\cite{M} realized that certain geometric ideas are very powerful
in tackling difficult arithmetic problems. In particular, Minkowski~\cite{M} proved that in a number field $k$ of degree $d>1$ and discriminant $\Delta$, every ideal class can be represented by an integral ideal of norm less than $|\Delta|^{1/2}$. His proof relied on two ideas. First, the natural embedding $k\hookrightarrow k\otimes_\QQ\RR$ allows one to regard the ring of integers $\mo$ as a lattice in $\RR^d$ of covolume $|\Delta|^{1/2}$. Second, a lattice in $\RR^d$ contains a nonzero lattice point in a convex body symmetric about the origin\footnote{that is, a convex subset of $\RR^d$ invariant under multiplication by $-1$}, as long as the volume of the body exceeds $2^d$ times the covolume of the lattice. The second idea was extended by Blichfeldt~\cite{B2} and van der Corput~\cite{C} to exhibit more lattice points in larger convex bodies. It leads to the following estimate that we state partly for motivation, partly as a technical ingredient for our investigations. For a modern exposition of the quoted results, see \cite[Ch.~2, \S 5.1 \& \S 7.2]{GL}.

\begin{theorem}[Minkowski~\cite{M}, Blichfeldt~\cite{B2}, van der Corput~\cite{C}]\label{thm1} Let $\mn\subset\mo$ be a nonzero ideal, and let $\cB\subset k\otimes_\QQ\RR$ be a convex body symmetric about the origin. Then
\[|\mn\cap\cB|\geq\frac{\vol(\cB)}{2^d|\Delta|^{1/2}[\mo:\mn]}.\]
\end{theorem}

Blichfeldt~\cite{B2} also established an upper bound of similar quality in the case when $\mn\cap\cB$ contains $d$ linearly independent vectors.

\begin{theorem}[Blichfeldt~\cite{B2}]\label{thm1b} Let $\mn\subset\mo$ be a nonzero ideal, and let $\cB\subset k\otimes_\QQ\RR$ be a convex body symmetric about the origin. Assume that $\mn\cap\cB$ contains $d$ linearly independent vectors. Then
\[|\mn\cap\cB|\leq\frac{(d+1)!\vol(\cB)}{|\Delta|^{1/2}[\mo:\mn]}.\]
\end{theorem}

In fact Blichfeldt proved a more general result, namely Theorem~\ref{thm5} in Section~\ref{sect3}. The original source \cite{B2} is an account of an AMS Sectional Meeting held in 1920 (written by B.~A.~Bernstein), so it does not contain any proof. What is worse, we could only find sketches of the proof in the literature. Hence we include a detailed proof in Section~\ref{sect3}, without claiming any originality.

Our principal goal in this paper is to provide an upper bound for $|\mn\cap\cB|$ in the complementary case when $\mn\cap\cB$ does not contain $d$ linearly independent vectors. More precisely, with certain arithmetic applications in mind, we restrict ourselves to the special convex bodies considered by Minkowski~\cite{M} in his seminal work. They are the archimedean analogues of ideal lattices, and they are defined as follows. As before, let $k$ be a number field of degree $d>1$. Let $\Sigma:=\Hom(k,\ov{\QQ})$, and let $K$ be the compositum of the fields $\sigma(k)$ for $\sigma\in\Sigma$. Then $K/\QQ$ is a finite Galois extension whose Galois group $G:=\Gal(K/\QQ)$ acts transitively and faithfully on $\Sigma$. In this way, $G$ is a transitive permutation group of degree $d$. Fixing an embedding $\ov{\QQ}\hookrightarrow\CC$, we can think of the elements of $\Sigma$ as the embeddings $\sigma:k\hookrightarrow\CC$, and we can identify $k\otimes_\QQ\RR$ with the set of column vectors $(z_\sigma)\in\CC^\Sigma$ satisfying $z_{\ov\sigma}=\ov{z_\sigma}$ for all $\sigma\in\Sigma$. See \cite[Ch.~I, \S 5]{N} for more details. Let $(B_\sigma)$ be a collection of positive numbers such that $B_{\ov\sigma}=B_\sigma$ for all $\sigma\in\Sigma$. We shall focus on convex bodies of the form
\begin{equation}\label{eq1}
\cB:=\left\{(z_\sigma)\in\CC^\Sigma:\text{$z_{\ov\sigma}=\ov{z_\sigma}$ and $|z_\sigma|\leq B_\sigma$ for all $\sigma\in\Sigma$}\right\},
\end{equation}
and we note for later reference that
\begin{equation}\label{eq19}
\vol(\cB)\asymp_d\prod_{\sigma\in\Sigma}B_\sigma.
\end{equation}
Here and later, the symbols $\ll_d$, $\gg_d$, $\asymp_d$ have their usual meaning in analytic number theory: $X\ll_d Y$ (resp. $Y\gg_d X$) means that $|X|\leq CY$ holds for an absolute constant $C>0$ depending only on $d$, while $X\asymp_d Y$ abbreviates $X\ll_d Y\ll_d X$.

\begin{theorem}\label{thm3b} Let $\mn\subset\mo$ be a nonzero ideal, and let $\cB\subset k\otimes_\QQ\RR$ be a convex body of the form \eqref{eq1}. Let $m$ be the maximal number of linearly independent lattice vectors contained in $\mn\cap\cB$. If $m<d$, then
\begin{equation}\label{eq2b}|\mn\cap\cB|\ll_d|\Delta|^{\min\left(\frac{1}{2},\frac{m}{2d-2m}\right)}.\end{equation}
Further, if $m<d$ and $G$ is $2$-homogeneous\footnote{that is, $G$ acts transitively on the $2$-element subsets of $\Sigma$}, then
\begin{equation}\label{eq3b}|\mn\cap\cB|\ll_d|\Delta|^{\frac{m}{2d-2}}.\end{equation}
\end{theorem}

Theorems~\ref{thm1b} and \ref{thm3b} yield a practical estimate for the number of elements of $k$ which are bounded in every archimedean and non-archimedean valuation of $k$.

\begin{corollary}\label{cor1b} Let $\mn\subset\mo$ be a nonzero ideal, and let $\cB\subset k\otimes_\QQ\RR$ be a convex body of the form \eqref{eq1}. Then
\begin{equation}\label{eq4b}
|\mn\cap\cB|\ll_d|\Delta|^{1/2}+\frac{\vol(\cB)}{|\Delta|^{1/2}[\mo:\mn]}.
\end{equation}
\end{corollary}

By combining Theorems~\ref{thm1} and \ref{thm3b}, we see that if the volume of our convex body is sufficiently large compared to the covolume of our ideal lattice, then the intersection contains several linearly independent lattice vectors.

\begin{corollary}\label{cor1} Let $\mn\subset\mo$ be a nonzero ideal, and let $\cB\subset k\otimes_\QQ\RR$ be a convex body of the form \eqref{eq1}. Let $m$ be the maximal number of linearly independent lattice vectors contained in $\mn\cap\cB$. If $m<d$, then
\begin{equation}\label{eq2}\vol(\cB)\ll_d|\Delta|^{\min\left(1,\frac{d}{2d-2m}\right)}[\mo:\mn].\end{equation}
Further, if $m<d$ and $G$ is $2$-homogeneous, then
\begin{equation}\label{eq3}\vol(\cB)\ll_d|\Delta|^{\frac{d-1+m}{2d-2}}[\mo:\mn].\end{equation}
\end{corollary}

If $m=0$, then \eqref{eq2b} and \eqref{eq3b} are trivial, while \eqref{eq2} and \eqref{eq3} boil down to the Minkowski bound
$\vol(\cB)\ll_d|\Delta|^{1/2}[\mo:\mn]$. If $m=1$ or $m=d-1$, then \eqref{eq2b} and \eqref{eq3b} (resp. \eqref{eq2} and \eqref{eq3}) are identical. For $2\leq m\leq d-2$, the bound \eqref{eq3b} is stronger than \eqref{eq2b} (resp. \eqref{eq3} is stronger than \eqref{eq2}), but its scope is restricted by the assumption that $G$ is $2$-homogeneous. The list of finite $2$-homogeneous groups is known by the work of many people, in particular by the classification of finite simple groups. For further details and references, see \cite[Prop.~3.1]{K}, \cite[Th.~5.3]{C2}, \cite[p.~198]{H}.

\begin{corollary}\label{cor2} Let $\mn\subset\mo$ be a nonzero ideal, and let $\cB\subset k\otimes_\QQ\RR$ be a convex body of the form \eqref{eq1}.
If $\cB$ does not contain a lattice basis of $\mn$, then $\vol(\cB)\ll_d|\Delta|[\mo:\mn]$.
\end{corollary}

Interestingly, when $k$ is totally real, the conclusion of Corollary~\ref{cor2} also follows from a celebrated result of McMullen~\cite[Th.~4.1]{M3} proved by topological arguments. In another direction, when the radii $B_\sigma$ are equal, the conclusion of Corollary~\ref{cor2} says that the last successive minimum\footnote{we understand successive minima with respect to the closed Euclidean unit ball centered at the origin} of $\mn$ is $\ll_d|\Delta|^{1/d}[\mo:\mn]^{1/d}$. For $\mn=\mo$, this bound was deduced earlier by Bhargava et al.~\cite[Th.~1.6]{B} with a more direct approach. We will return to these connections in Section~\ref{sect5}. In fact we can control, to some extent, all successive minima of ideal lattices.

\begin{theorem}\label{thm4b} Let $\lambda_1\leq\dotsb\leq\lambda_d$ be the successive minima of a nonzero ideal $\mn\subset\mo$ embedded as a lattice in $k\otimes_\QQ\RR$. Then for all $m\in\{1,\dotsc,d-1\}$ we have
\begin{align}\label{eq21}
\lambda_1\cdots\lambda_m&\gg_d|\Delta|^{\max\left(0,\frac{m}{d}-\frac{1}{2}\right)}[\mo:\mn]^{\frac{m}{d}};\\
\label{eq20}
\lambda_{m+1}\lambda_{m+2}\cdots\lambda_d&\ll_d|\Delta|^{\min\left(\frac{1}{2},1-\frac{m}{d}\right)}[\mo:\mn]^{1-\frac{m}{d}}.
\end{align}
If $G$ is $2$-homogeneous, then the exponents of $|\Delta|$ in \eqref{eq21} and \eqref{eq20} can be improved to
$\frac{m(m-1)}{2d(d-1)}$ and $\frac{(d-m)(d+m-1)}{2d(d-1)}$, respectively.
\end{theorem}

The example $k=\QQ(p^{1/d})$ mentioned by Bhargava et al.~below their \cite[Th.~1.6]{B} shows that the $2$-homogeneous case of Theorem~\ref{thm4b} cannot be improved in general. Indeed, if $p>d>1$ are prime numbers and $\mn=\mo$, then $G\cong\mathrm{Aff}(\FF_d)\cong(\ZZ/d\ZZ)\rtimes(\ZZ/d\ZZ)^\times$ is sharply $2$-transitive, while $\lambda_m\asymp_d|\Delta|^\frac{m-1}{d(d-1)}$ holds for all $m\in\{1,\dotsc,d\}$. The last relation follows from the straightforward upper bound $\lambda_m\ll_d p^\frac{m-1}{d}$ combined with $|\Delta|\asymp_d p^{d-1}$ and Minkowski's result \eqref{eq28} quoted below. The same example also shows that Corollary~\ref{cor2} cannot be improved in general. In contrast, the sharpness of \eqref{eq2b}--\eqref{eq3b} and \eqref{eq21}--\eqref{eq20} is less clear to us.

Theorem~\ref{thm4b} readily yields two-sided bounds for individual successive minima, extending the result of Bhargava et al.~\cite[Th.~1.6]{B} mentioned in the previous paragraph.

\begin{corollary}\label{cor3} Let $\lambda_1\leq\dotsb\leq\lambda_d$ be the successive minima of a nonzero ideal $\mn\subset\mo$ embedded as a lattice in $k\otimes_\QQ\RR$. Then for all $m\in\{1,\dotsc,d\}$ we have
\begin{align}\label{eq26}
\Delta^{\max\left(0,\frac{1}{d}-\frac{1}{2m}\right)}[\mo:\mn]^{\frac{1}{d}}&\ll_d\lambda_m
\ll_d\Delta^{\min\left(\frac{1}{2d-2m+2},\frac{1}{d}\right)}[\mo:\mn]^{\frac{1}{d}}&&\text{in general};\\
\label{eq27}
\Delta^{\frac{m-1}{2d(d-1)}}[\mo:\mn]^{\frac{1}{d}}&\ll_d\lambda_m
\ll_d\Delta^{\frac{d+m-2}{2d(d-1)}}[\mo:\mn]^{\frac{1}{d}}&&\text{if $G$ is $2$-homogeneous}.
\end{align}
\end{corollary}

To form an idea of the accuracy of \eqref{eq27}, it is instructive to observe that the two sides differ by a factor of $\Delta^{\frac{1}{2d}}$. Moreover, the product of the left hand side over $m\in\{1,\dotsc,d\}$ equals $\Delta^{\frac{1}{4}}[\mo:\mn]$, while the same for the right hand side equals $\Delta^{\frac{3}{4}}[\mo:\mn]$. This should be compared with the product of the $\lambda_m$'s, which by Minkowski's theorem \cite[p.~124, Th.~3]{GL} is
\begin{equation}\label{eq28}\lambda_1\cdots\lambda_d\asymp_d|\Delta|^\frac{1}{2}[\mo:\mn].\end{equation}

The proof of Theorem~\ref{thm3b} combines group theory, ramification theory, and the geometry of numbers. The main idea is to obtain an upper bound for $|\mn\cap\cB|$ by projecting $\mn\cap\cB$ onto well-chosen ``coordinate subspaces'' $\RR^S$ of $\CC^\Sigma$ for $S\subset\Sigma$, and then compare it with the lower bound of Theorem~\ref{thm1}. We make sure that the projections of $\mn\cap\cB$ generate lattices in their ambient spaces $\RR^S$, and then we succeed by bounding from below the product of covolumes of those lattices. The proof of Theorem~\ref{thm4b} is similar, but it focuses on successive minima in place of lattice point counts. In order to formulate the key arithmetic ingredient of both proofs, Theorem~\ref{thm3} below, we need to introduce further notation.

For a nonzero prime ideal $\mpr\subset\mo$ dividing a rational prime $p$, let $e_\mpr$ (resp. $f_\mpr$) denote the ramification index (resp. inertia degree) of the local field extension $k_\mpr/\QQ_p$. By \cite[Ch.~III, \S 2]{N}, the exponent of $\mpr$ in the different ideal of $\mo$ equals $e_\mpr-1$ when $p\nmid e_\mpr$, and it lies between $e_\mpr$ and $e_\mpr-1+v_\mpr(e_\mpr)$ when $p\mid e_\mpr$ (which can only occur for $p\leq d$). Therefore, the \emph{tame discriminant} $\Dt$, defined as
\begin{equation}\label{eq5}
\Dt:=\prod_p p^{d-f_p}\quad\text{with}\quad f_p:=\sum_{\mpr\mid p}f_\mpr,
\end{equation}
divides the discriminant $\Delta$, and it satisfies
\begin{equation}\label{eq4}
|\Delta|<2^{d^3}\Dt.
\end{equation}
The last bound is rather crude, and it can be verified as follows. The ratio $\Delta/\Dt$ divides the norm of the ideal $\prod_{p\leq d}\prod_{\mpr\mid p}\mpr^{v_\mpr(e_\mpr)}$, which is a divisor of the principal ideal $\prod_{p\leq d}\prod_{\mpr\mid p}(e_\mpr)$. Therefore,
\[\frac{|\Delta|}{\Dt}\leq\prod_{p\leq d}\prod_{\mpr\mid p}e_\mpr^d<\prod_{p\leq d}2^{d\sum_{\mpr\mid p}e_\mpr}<2^{d^3}.\]

\begin{theorem}\label{thm3} Let $\mn\subset\mo$ be a nonzero ideal, and let $m\in\{1,\dotsc,d\}$. For any $m$-subsets $X\subset\mn$ and $S\subset\Sigma$,
\begin{equation}\label{eq6}
\prod_{g\in G}{\det}^2(\sigma(x))^{\sigma\in gS}_{x\in X}\quad\text{is divisible by}\quad\Dt^{|G|\max\left(0,\frac{2m}{d}-1\right)}[\mo:\mn]^{|G|\frac{2m}{d}}.
\end{equation}
If $G$ is $2$-homogeneous, then the exponent of $\Dt$ can be improved to $|G|\frac{m(m-1)}{d(d-1)}$.
\end{theorem}

Note that $d$ divides $|G|$, and also $\binom{d}{2}$ divides $G$ when $G$ is $2$-homogeneous, so the exponents of $\Dt$ and $[\mo:\mn]$ are nonnegative integers. The next theorem is very similar to the $2$-homogeneous case of Theorem~\ref{thm3}. We do not need it for the proof of Theorem~\ref{thm3b}, but we present it for its intrinsic beauty and interest.

\begin{theorem}\label{thm4} Let $\mn\subset\mo$ be a nonzero ideal, and let $m\in\{2,\dotsc,d\}$. For any $m$-subset $X\subset\mn$,
\begin{equation}\label{eq8}
\prod_{\substack{S\subset\Sigma\\|S|=m}}{\det}^2(\sigma(x))^{\sigma\in S}_{x\in X}\quad\text{is divisible by}\quad\Dt^{\binom{d-2}{m-2}}[\mo:\mn]^{2\binom{d-1}{m-1}}.
\end{equation}
\end{theorem}

The determinants in \eqref{eq6} and \eqref{eq8} are only defined up to a factor of $\pm 1$, because we have not specified any ordering on $X$ and $S$. However, their squares are well-defined. If $m=d$, then Theorems~\ref{thm3} and \ref{thm4} follow from the fact that either $\det(\sigma(x))^{\sigma\in\Sigma}_{x\in X}$ is zero, or it equals the covolume of a full rank sublattice of $\mn$. Another relatively simple special case is when $\mn=\mo$ and $X=\{1,x,\ldots,x^{m-1}\}$ for some $x\in\mo$. Then, Theorem~\ref{thm4} and the $2$-homogeneous case of Theorem~\ref{thm3} are consequences of the Vandermonde determinant formula and the definition of the (usual) discriminant $\Delta$ of $k$. Not surprisingly, we shall only use the divisibility conclusion when the participating determinants are nonzero. On the other hand, it seems to be an interesting and difficult problem to characterize the vanishing of these determinants. One result in this direction is Chebotarev's theorem from 1926: if $p$ is a prime, $k$ is the $p$-th cyclotomic field, and the elements of $X$ are $p$-th roots of unity, then none of these determinants vanish (see \cite{T} for a proof and for useful references). Another result is the following simple observation: if $k$ contains a proper subfield $k'$ with $m=[k:k']$, and the $m$-subset $X\subset k$ is linearly dependent over $k'$, then there is an $m$-subset $S\subset\Sigma$ such that all embeddings $\sigma\in S$ coincide on $k'$, whence $\det(\sigma(x))^{\sigma\in S}_{x\in X}=0$. Motivated by this example, we ask the following question:

\begin{question}
Assume that $X\subset k$ and $S\subset\Sigma$ satisfy $|X|=|S|$ and $\det(\sigma(x))^{\sigma\in S}_{x\in X}=0$. Does there exist a subfield $k'$ of $k$ such that $X$ is linearly dependent over $k'$, and all embeddings $\sigma\in S$ coincide on $k'$?
\end{question}

If $X$ is of size $m$ and $G$ is $m$-homogeneous (e.g.\ when $G=S_d$ or $G=A_d$), then the answer to this question is affirmative. Indeed, in this case, the vanishing of one $m\times m$ minor of $\det(\sigma(x))^{\sigma\in\Sigma}_{x\in X}$ implies the vanishing of all $m\times m$ minors, which can happen if and only if $X$ is linearly dependent over $\QQ$.

\begin{acknowledgement}
We are grateful to the referees for their careful reading and valuable comments. We also thank P\'eter P\'al P\'alfy and Gergely Z\'abr\'adi for helpful discussions.
\end{acknowledgement}

\section{Non-archimedean investigations}

In this section, we prove Theorems~\ref{thm3} and \ref{thm4}. The two sides of \eqref{eq6} and \eqref{eq8} are rational integers, hence it suffices to show, for every rational prime $p$, that the exponent of $p$ is at least as large on the left hand side as on the right hand side (with the convention that the $p$-exponent of zero is infinity).

We fix $p$ and an embedding $\ov{\QQ}\hookrightarrow\ov{\QQ_p}$, then we can think of the elements of $\Sigma$ as the embeddings $\sigma:k\hookrightarrow\ov{\QQ_p}$. For each $\sigma\in\Sigma$, there is a unique prime ideal $\mpr\mid p$ and a unique $\QQ_p$-linear extension $\tilde\sigma:k_\mpr\hookrightarrow\ov{\QQ_p}$ of $\sigma$. Denoting by $I_\mpr$ the set of $\sigma$'s corresponding to a given $\mpr$, the extension map $\sigma\mapsto\tilde\sigma$ is a bijection $I_\mpr\overset{\sim}\to\Hom_{\QQ_p}(k_\mpr,\ov{\QQ_p})$ with inverse being the restriction map. In particular, $I_\mpr$ is a $\Gal(\ov{\QQ_p}/\QQ_p)$-orbit on $\Sigma$ of cardinality $[k_\mpr:\QQ_p]=e_\mpr f_\mpr$. Let $v_p$ be the unique additive valuation on $\ov{\QQ_p}$ extending the normalized additive valuation on $\QQ_p$, and let $v_\mpr$ be the normalized additive valuation on $k_\mpr$. By ``normalized'' we mean that $v_p(\QQ_p^\times)=\ZZ$ and $v_\mpr(k_\mpr^\times)=\ZZ$. Then we have the important identity
\begin{equation}\label{eq11}
v_p(\tilde\sigma(x))=\frac{1}{e_\mpr}v_\mpr(x),\qquad\tilde\sigma\in\Hom_{\QQ_p}(k_\mpr,\ov{\QQ_p}),\qquad x\in k_\mpr^\times.
\end{equation}
See \cite[Ch.~II, \S 8]{N} for more details.
Let $l_\mpr$ be the maximal unramified subextension of $k_\mpr/\QQ_p$, then
\[[k_\mpr:l_\mpr]=e_\mpr\qquad\text{and}\qquad [l_\mpr:\QQ_p]=f_\mpr.\]
Identifying $I_\mpr$ with $\Hom_{\QQ_p}(k_\mpr,\ov{\QQ_p})$ as above, we can break up $I_\mpr$ into $f_\mpr$ subsets $I_{\mpr,l}$ of equal size $e_\mpr$ according to how $l_\mpr$ gets embedded into $\ov{\QQ_p}$. In the end, two elements of $\Sigma$ belong to the same subset $I_{\mpr,l}$ if and only if they induce the same non-archimedean valuation $|\cdot|_\mpr$ on $k$ and their $\QQ_p$-linear extensions agree on $l_\mpr$; we shall call two such elements of $\Sigma$ \emph{inertially equivalent}.

The proofs of Theorems~\ref{thm3} and \ref{thm4} rely on the key observation that the $p$-adic valuation of the participating determinants can be estimated in terms of the inertial equivalence classes $I_{\mpr,l}$.

\begin{proposition}\label{prop1} Let $\mn\subset\mo$ be a nonzero ideal, and let $m\in\{1,\dotsc,d\}$. For any $m$-subsets $X\subset\mn$ and $S\subset\Sigma$,
\begin{equation}\label{eq9}
v_p\Bigl({\det}^2(\sigma(x))^{\sigma\in S}_{x\in X}\Bigr)\geq
\sum_{\mpr\mid p}\frac{1}{e_\mpr}\sum_{l=1}^{f_\mpr}s_{\mpr,l}\bigl(2v_\mpr(\mn)+s_{\mpr,l}-1\bigr),
\end{equation}
where $s_{\mpr,l}$ abbreviates $|S\cap I_{\mpr,l}|$, and $v_\mpr(\mn)$ stands for the exponent of $\mpr$ in $\mn$.
\end{proposition}

\begin{proof} We recall that $K$ is the compositum of the fields $\sigma(k)$ for $\sigma\in\Sigma$, and we write $\tilde K$ for the extension of $\QQ_p$ generated by $K$. We denote by $\tilde d$ the degree $[\tilde K:\QQ_p]$, and by $\tilde\mo$ the ring of integers of $\tilde K$. We shall think of $\tilde\mo^m$ as the set of column vectors of length $m$ with entries in $\tilde\mo$.

The $m$-set $S\subset\Sigma$ is partitioned into the $s_{\mpr,l}$-sets $S_{\mpr,l}:=S\cap I_{\mpr,l}$.
Accordingly, the $m\times m$ matrix $A:=(\sigma(x))^{\sigma\in S}_{x\in X}$ decomposes into the $s_{\mpr,l}\times m$ blocks
$A_{\mpr,l}:=(\sigma(x))^{\sigma\in S_{\mpr,l}}_{x\in X}$. Strictly speaking, these matrices are only defined up to a permutation of the rows and the columns, but this ambiguity disappears once we choose an ordering of the rows and the columns.

We shall assume that $\det A\neq 0$, for otherwise \eqref{eq9} is trivial. The natural isomorphism from $\tilde\mo^m$ to $\prod_\mpr\prod_l\tilde\mo^{s_{\mpr,l}}$ maps $A\tilde\mo^m$ into $\prod_\mpr\prod_l A_{\mpr,l}\tilde\mo^m$, hence it induces a surjective homomorphism from $\tilde\mo^m/A\tilde\mo^m$ onto $\prod_\mpr\prod_l(\tilde\mo^{s_{\mpr,l}}/A_{\mpr,l}\tilde\mo^m)$. In particular,
\[v_p\bigl([\tilde\mo^m:A\tilde\mo^m]\bigr)\geq\sum_{\mpr\mid p}\sum_{l=1}^{f_\mpr}v_p\bigl([\tilde\mo^{s_{\mpr,l}}:A_{\mpr,l}\tilde\mo^m]\bigr).\]
The left hand side equals $\tilde d\cdot v_p(\det A)$, hence \eqref{eq9} will follow if we can show that
\begin{equation}\label{eq10}
v_p\bigl([\tilde\mo^{s_{\mpr,l}}:A_{\mpr,l}\tilde\mo^m]\bigr)\geq\frac{\tilde d}{e_\mpr}s_{\mpr,l}\left(v_\mpr(\mn)+\frac{s_{\mpr,l}-1}{2}\right).
\end{equation}

Let us fix $\mpr\mid p$ and $l\in\{1,\dotsc,f_\mpr\}$. We shall assume that $S_{\mpr,l}$ is not empty, for otherwise \eqref{eq10} is trivial. We write \begin{equation}\label{eq14}
t:=s_{\mpr,l}\qquad\text{and}\qquad B:=A_{\mpr,l}
\end{equation}
to simplify notation, and we list the elements of $S_{\mpr,l}$ as $\{\sigma_1,\dotsc,\sigma_t\}$. By \eqref{eq11}, we have
\begin{equation}\label{eq12}
v_p(\tilde\sigma_i(x))=\frac{1}{e_\mpr}v_\mpr(x),\qquad i\in\{1,\dotsc,t\},\qquad x\in k_\mpr^\times.
\end{equation}
We also list the elements of $X$ as $\{x_1,\dotsc,x_m\}$ in such a way that
\[v_\mpr(\mn)\leq v_\mpr(x_1)\leq\dotsb\leq v_\mpr(x_m).\]
In particular, $v_p$ is constant on each column of
\[B=\begin{pmatrix}
\sigma_1(x_1) & \cdots & \sigma_1(x_m) \\
\vdots  & \ddots & \vdots  \\
\sigma_t(x_1)& \cdots & \sigma_t(x_m)
\end{pmatrix},\]
and it is non-decreasing from left to right.
As the $\sigma_i$'s are inertially equivalent, their $\QQ_p$-linear extensions $\tilde\sigma_i$ coincide on $l_\mpr$, and we can identify $l_\mpr$ with its image in $\tilde K$ via any of these embeddings. A nice feature resulting from this identification is that the $\tilde\sigma_i$'s are $l_\mpr$-linear, not just $\QQ_p$-linear.

We are ready to prove \eqref{eq10}. We shall use the fact that the left hand side of \eqref{eq10}, which is $[\tilde\mo^t:B\tilde\mo^m]$ in our new notation \eqref{eq14}, remains unchanged if we multiply $B$ by elements of $\GL_m(\tilde\mo)$ on the right and by elements of $\GL_t(\tilde\mo)$ on the left.
Writing $\mo_{l_\mpr}$ (resp. $\mo_{k_\mpr}$) for the ring of integers of $l_\mpr$ (resp. $k_{\mpr}$), we shall also utilize the fact that the group of units $\mo_{l_\mpr}^\times$ contains a full set of representatives for the nonzero residue classes modulo $\mpr\mo_{k_\mpr}$ in $\mo_{k_\mpr}$. This is because the residue fields of $l_{\mpr}$ and $k_{\mpr}$ have equal cardinality $p^{f_\mpr}$.

First, we perform invertible elementary column operations over $\mo_{l_\mpr}$ in order to increase the additive valuations of the columns of $B$.
Specifically, we run the following algorithm:
\begin{enumerate}[\indent 1.]
\item Set $j=1$.
\item For each $j'\in\{j+1,\dotsc,m\}$, if $v_\mpr(x_{j'})=v_\mpr(x_j)$, then choose $w\in\mo_{l_\mpr}^\times$ such that $v_\mpr(x_{j'}-wx_j)>v_\mpr(x_j)$ and replace $x_{j'}$ by $x_{j'}-wx_j$.
\item Reorder $(x_{j+1},\dotsc,x_m)$ in such a way that $v_\mpr$ is non-decreasing on the new sequence.
\item Replace $j$ by $j+1$.
\item If $j<m$, then go to the second step; otherwise, finish.
\end{enumerate}
We end up with a matrix
\[C=\begin{pmatrix}
\tilde\sigma_1(y_1) & \cdots & \tilde\sigma_1(y_m) \\
\vdots  & \ddots & \vdots  \\
\tilde\sigma_t(y_1)& \cdots & \tilde\sigma_t(y_m)
\end{pmatrix}\]
with $y_1,\dotsc,y_m\in\mo_{k_\mpr}$ such that
\[v_\mpr(\mn)\leq v_\mpr(y_1)<\dotsb<v_\mpr(y_m).\]
In particular, $v_\mpr(y_j)\geq v_\mpr(\mn)+j-1$ for all $j\in\{1,\dotsc,m\}$.

Second, we perform invertible elementary row operations over $\tilde\mo$ to transform $C$ into
\[D=\begin{pmatrix} z_{1,1} & z_{1,2} & \cdots &\ z_{1,t}&\cdots & z_{1,m}\\
0 & z_{2,2} & \cdots & z_{2,t}& \cdots & z_{2,m}\\
\vdots & \ddots & \ddots & \vdots & \ddots & \vdots\\
0 & \ldots & 0 & z_{t,t} &\ldots & z_{t,m}\\
\end{pmatrix}\]
with $z_{i,j}\in\tilde\mo$ such that (cf.\ \eqref{eq12})
\[v_p(z_{i,j})\geq\frac{1}{e_\mpr}\bigl(v_\mpr(\mn)+j-1\bigr),\qquad i\leq j.\]
In particular, $D\tilde\mo^m$ is a subgroup of $\tilde\mn_1\times\dotsb\times\tilde\mn_t$, where
\[\tilde\mn_i:=\left\{z\in\tilde\mo:v_p(z)\geq\frac{1}{e_\mpr}\bigl(v_\mpr(\mn)+i-1\bigr)\right\},\qquad i\in\{1,\dotsc,t\}.\]
This implies, using that $e_\mpr$ divides the ramification degree of the local field extension $\tilde K/\QQ_p$,
\begin{equation}\label{eq13}
v_p\bigl([\tilde\mo^t:D\tilde\mo^m]\bigr)\geq\sum_{i=1}^t v_p\bigl([\tilde\mo:\tilde\mn_i]\bigr)
=\sum_{i=1}^t\frac{\tilde d}{e_\mpr}\bigl(v_\mpr(\mn)+i-1\bigr).
\end{equation}

The inequalities \eqref{eq13} and \eqref{eq10} are equivalent, because their left hand sides are equal, and their right hand sides are also equal (cf.\ \eqref{eq14}). The proof of Proposition~\ref{prop1} is complete.
\end{proof}

\begin{proof}[Proof of Theorem~\ref{thm3}]
For any $g\in G$, it follows from Proposition~\ref{prop1} that
\[v_p\Bigl({\det}^2(\sigma(x))^{\sigma\in gS}_{x\in X}\Bigr)\geq
\sum_{\mpr\mid p}\frac{1}{e_\mpr}\sum_{l=1}^{f_\mpr}\sum_{\sigma\in I_{\mpr,l}}1_{gS}(\sigma)
\left(2v_\mpr(\mn)+\sum_{\sigma'\in I_{\mpr,l}\setminus\{\sigma\}}1_{gS}(\sigma')\right).\]
We average both sides over $g\in G$, utilizing that $G$ acts transitively and faithfully on $\Sigma$.  For any $\sigma\in\Sigma$,
we obtain readily that
\begin{equation}\label{eq18}
\frac{1}{|G|}\sum_{g\in G}1_{gS}(\sigma)=\frac{1}{|G|}\sum_{g\in G}1_{S}(g^{-1}\sigma)=\frac{|S|}{d}=\frac{m}{d}.
\end{equation}
As a consequence, for any distinct $\sigma,\sigma'\in\Sigma$, we see that
\begin{equation}\label{eq7}
\frac{1}{|G|}\sum_{g\in G}1_{gS}(\sigma)1_{gS}(\sigma')\geq
\frac{1}{|G|}\sum_{g\in G}\bigl(1_{gS}(\sigma)+1_{gS}(\sigma')-1\bigr)=\frac{2m}{d}-1.
\end{equation}
This bound is trivial when $m<d/2$, in which case we shall only use that the left hand side is nonnegative. Combining these inequalities and noting that $|I_{\mpr,l}|=e_\mpr$, we infer that
\[\frac{1}{|G|}\sum_{g\in G}v_p\Bigl({\det}^2(\sigma(x))^{\sigma\in gS}_{x\in X}\Bigr)\geq
\sum_{\mpr\mid p}f_\mpr\left(v_\mpr(\mn)\frac{2m}{d}+(e_\mpr-1)\max\left(0,\frac{2m}{d}-1\right)\right).\]
Now from $[\mo:\mpr]=p^{f_\mpr}$ it is clear that
\[\sum_{\mpr\mid p}f_\mpr v_\mpr(\mn)=v_p\bigl([\mo:\mn]\bigr),\]
while \eqref{eq5} implies that
\[\sum_{\mpr\mid p}f_\mpr(e_\mpr-1)=d-f_p=v_p(\Dt).\]
Therefore, the last inequality can be rewritten as
\[\frac{1}{|G|}\sum_{g\in G}v_p\Bigl({\det}^2(\sigma(x))^{\sigma\in gS}_{x\in X}\Bigr)\geq
\frac{2m}{d}v_p\bigl([\mo:\mn]\bigr)+\max\left(0,\frac{2m}{d}-1\right)v_p(\Dt).\]
The rational prime $p$ was arbitrary here, so we have proved \eqref{eq6}.

If $G$ is $2$-homogeneous, then we can improve \eqref{eq7} to
\[\frac{1}{|G|}\sum_{g\in G}1_{gS}(\sigma)1_{gS}(\sigma')=\frac{1}{|G|}\sum_{g\in G}1_{S}(g^{-1}\sigma)1_{S}(g^{-1}\sigma')
=\frac{\binom{|S|}{2}}{\binom{d}{2}}=\frac{m(m-1)}{d(d-1)}.\]
As a result, we can replace $\max\left(0,\frac{2m}{d}-1\right)$ by $\frac{m(m-1)}{d(d-1)}$ in the subsequent argument, and hence also in \eqref{eq6}. The proof of Theorem~\ref{thm3} is complete.
\end{proof}

\begin{proof}[Proof of Theorem~\ref{thm4}]
For any $m$-subset $S\subset\Sigma$, it follows from Proposition~\ref{prop1} that
\[v_p\Bigl({\det}^2(\sigma(x))^{\sigma\in S}_{x\in X}\Bigr)\geq
\sum_{\mpr\mid p}\frac{1}{e_\mpr}\sum_{l=1}^{f_\mpr}\sum_{\sigma\in I_{\mpr,l}}1_{S}(\sigma)
\left(2v_\mpr(\mn)+\sum_{\sigma'\in I_{\mpr,l}\setminus\{\sigma\}}1_{S}(\sigma')\right).\]
We sum both sides over all $m$-subsets $S\subset\Sigma$, using that
\begin{align*}
\sum_{\substack{S\subset\Sigma\\|S|=m}}1_S(\sigma)&=\binom{d-1}{m-1}\quad\text{for any $\sigma\in\Sigma$};\\
\sum_{\substack{S\subset\Sigma\\|S|=m}}1_S(\sigma)1_S(\sigma')&=\binom{d-2}{m-2}\quad\text{for any distinct $\sigma,\sigma'\in\Sigma$.}
\end{align*}
From here we proceed as in the proof of Theorem~\ref{thm3}, and conclude
\[\sum_{\substack{S\subset\Sigma\\|S|=m}}v_p\Bigl({\det}^2(\sigma(x))^{\sigma\in S}_{x\in X}\Bigr)\geq
2\binom{d-1}{m-1}v_p\bigl([\mo:\mn]\bigr)+\binom{d-2}{m-2}v_p(\Dt).\]
The rational prime $p$ was arbitrary here, so the proof of Theorem~\ref{thm4} is complete.
\end{proof}

\section{Archimedean investigations}\label{sect3}

In this section, we prove Theorems~\ref{thm3b}--\ref{thm4b} and Corollaries~\ref{cor1b}--\ref{cor3}. We shall combine Theorems~\ref{thm1} and \ref{thm3} with the following lesser known result of Blichfeldt~\cite{B2}, of which Theorem~\ref{thm1b} is a special case.

\begin{theorem}[Blichfeldt~\cite{B2}]\label{thm5} Let $\Lambda\subset\RR^m$ be a lattice, and let $\cC\subset\RR^m$ be a convex body containing the origin. If $\Lambda\cap\cC$ contains $m$ linearly independent lattice vectors, then
\begin{equation}\label{eq15}
|\Lambda\cap\cC|\leq m!\frac{\vol(\cC)}{\det(\Lambda)}+m\leq(m+1)!\frac{\vol(\cC)}{\det(\Lambda)}.
\end{equation}
\end{theorem}

\begin{proof} The second inequality is clear by $\vol(\cC)\geq\det(\Lambda)/m!$, hence we focus on the first inequality. In this proof, a polytope (resp. simplex) will always mean a convex lattice polytope (resp. simplex) with vertices lying in $\Lambda$. For other terminology, we follow the book \cite{DRS}. Without loss of generality, $\cC$ is bounded. Then, by the initial assumptions on $\cC$, the convex hull of $\Lambda\cap\cC$ is an $m$-dimensional polytope, which can be decomposed into $m$-simplices according to \cite[Prop.~2.2.4]{DRS}. The corresponding triangulation of $\Lambda\cap\cC$ can be refined to a full triangulation by decomposing recursively the participating $m$-simplices into smaller $m$-simplices. Alternatively, one can obtain a full triangulation of $\Lambda\cap\cC$ by ordering its elements in such a way that no point belongs to the convex hull of previous points, and then taking the placing/pushing triangulation for that ordering. We fix a full triangulation of $\Lambda\cap\cC$, and we denote by $\cT$ the set of $m$-simplices that participate in it. We define a graph on $\cT$ by declaring that two elements of $\cT$ are connected by an edge if and only if their intersection is an $(m-1)$-simplex. One can show that this graph is connected, which forces
\[|\cT|\geq |\Lambda\cap\cC|-m.\]
For details, see \cite[Th.~2.6.1]{DRS}, \cite[Th.~3.2]{RS}, and their proofs. On the other hand, as $\cC$ is convex and each element of $\cT$ has volume at least $\det(\Lambda)/m!$, we also have
\[\vol(\cC)\geq\vol(\cup\cT)\geq\frac{\det(\Lambda)}{m!}|\cT|.\]
Combining these two bounds, we get the first inequality of \eqref{eq15}. As remarked earlier, the second inequality of \eqref{eq15} is straightforward, so the proof of Theorem~\ref{thm5} is complete.
\end{proof}

\begin{proof}[Proof of Theorem~\ref{thm3b}] If $m=0$, then \eqref{eq2b} and \eqref{eq3b} are trivial, so we shall assume that $0<m<d$. We write $V$ for the $\RR$-span of $\mn\cap\cB$, so that $V$ is an $m$-dimensional $\RR$-subspace of $k\otimes_\QQ\RR$, and $\mn\cap V$ is an $m$-dimensional lattice in $V$. We fix a basis $X\subset\mn$ of $\mn\cap V$, and we think of its elements as the columns of the $d\times m$ complex matrix $M:=(\sigma(x))^{\sigma\in\Sigma}_{x\in X}$. Strictly speaking, $M$ is only defined up to a permutation of the rows and the columns, but this ambiguity disappears once we choose an ordering of $\Sigma$ and $X$. By construction, the columns of $M$ are linearly independent over $\RR$, and we claim that they are also linearly independent over $\CC$. Indeed, if $c:X\to\CC$ satisfies $\sum_{x\in X}c(x)\sigma(x)=0$ for all $\sigma\in\Sigma$, then complex conjugating the equations and switching from $\sigma$ to $\ov{\sigma}$, we get that $\sum_{x\in X}\ov{c(x)}\sigma(x)=0$ for all $\sigma\in\Sigma$. As a result, the real and imaginary parts of $c(x)$ must vanish for all $x\in X$, which proves the claim. Hence $\rank(M)=m$, and there exists an $m$-subset $S\subset\Sigma$ such that $\det(\sigma(x))^{\sigma\in S}_{x\in X}\neq 0$. We fix $S\subset\Sigma$ along with $X\subset\mn$.

For any Galois automorphism $g\in G$, the image of $\det(\sigma(x))^{\sigma\in S}_{x\in X}$ under $g$ equals $\det(\sigma(x))^{\sigma\in gS}_{x\in X}$. Therefore, these $m\times m$ minors of $M$ are nonzero, and by \eqref{eq4} and Theorem~\ref{thm3} they satisfy
\begin{equation}\label{eq16}
\prod_{g\in G}\left|\det(\sigma(x))^{\sigma\in gS}_{x\in X}\right|\gg_d
|\Delta|^{|G|\max\left(0,\frac{m}{d}-\frac{1}{2}\right)}[\mo:\mn]^{|G|\frac{m}{d}}.
\end{equation}
Moreover, the exponent of $|\Delta|$ can be improved to $|G|\frac{m(m-1)}{2d(d-1)}$ when $G$ is $2$-homogeneous.

Fixing $g\in G$ for a moment, the multilinearity of the determinant shows that there is a choice of $\tilde\sigma\in\{\re(\sigma),\im(\sigma)\}$ for each $\sigma\in gS$ such that
\begin{equation}\label{eq17}
\left|\det(\sigma(x))^{\sigma\in gS}_{x\in X}\right|\leq 2^m\left|\det(\tilde\sigma(x))^{\sigma\in gS}_{x\in X}\right|.
\end{equation}
The left hand side is positive, hence the right hand side is also positive. Let $f:\CC^\Sigma\to\RR^{gS}$ be the product of the $\RR$-linear surjections $f_\sigma:\CC\to\RR$ given by
\[f_\sigma(z):=\begin{cases}
\re(z),&\sigma\in gS\ \ \text{and}\ \ \tilde\sigma=\re(\sigma);\\
\im(z),&\sigma\in gS\ \ \text{and}\ \ \tilde\sigma=\im(\sigma);\\
0,&\sigma\not\in gS.
\end{cases}\]
Tautologically, $\tilde\sigma=f_\sigma\circ\sigma$ holds for all $\sigma\in gS$, hence $f$ restricts to an $\RR$-linear isomorphism $V\overset{\sim}\to\RR^{gS}$, and $\Lambda:=f(\mn\cap V)$ is a lattice in $\RR^{gS}$ of covolume $\left|\det(\tilde\sigma(x))^{\sigma\in gS}_{x\in X}\right|$. In addition, $\cC:=f(\cB)$ is an $o$-symmetric convex body in $\RR^{gS}$, which lies in the orthotope $\prod_{\sigma\in gS}[-B_\sigma,B_\sigma]$ by \eqref{eq1}. Clearly, $\Lambda\cap\cC$ contains $f(\mn\cap\cB)$, which in turn contains $m$ linearly independent lattice vectors. Now we combine these observations with Theorem~\ref{thm5} and \eqref{eq17} to infer that
\[|\mn\cap\cB|\leq|\Lambda\cap\cC|\leq 4^m(m+1)!\frac{\prod_{\sigma\in gS}B_\sigma}{\left|\det(\sigma(x))^{\sigma\in gS}_{x\in X}\right|}.\]

We keep the two sides of the last inequality, and take their geometric mean over $g\in G$. Using also \eqref{eq19}, \eqref{eq18}, \eqref{eq16}, we obtain
\begin{equation}\label{eq23}
|\mn\cap\cB|\ll_d\frac{\vol(\cB)^\frac{m}{d}}{|\Delta|^{\max\left(0,\frac{m}{d}-\frac{1}{2}\right)}[\mo:\mn]^{\frac{m}{d}}}.
\end{equation}
Finally, we invoke Theorem~\ref{thm1} to estimate from above the right hand side in terms of the left hand side:
\begin{equation}\label{eq23b}
|\mn\cap\cB|\ll_d|\mn\cap\cB|^\frac{m}{d}|\Delta|^{\min\left(\frac{m}{2d},\frac{1}{2}-\frac{m}{2d}\right)}.
\end{equation}
This bound is equivalent to \eqref{eq2b} in the light of $0<m<d$. If $G$ is $2$-homogeneous, then the exponent of $|\Delta|$ can be improved to $\frac{m(m-1)}{2d(d-1)}$ in \eqref{eq23}, and to $\frac{m(d-m)}{2d(d-1)}$ in \eqref{eq23b}, so that the resulting bound is equivalent to \eqref{eq3b}. The proof of Theorem~\ref{thm3b} is complete.
\end{proof}

\begin{proof}[Proof of Corollary~\ref{cor1b}]
If $\mn\cap\cB$ contains $d$ linearly independent vectors, then \eqref{eq4b} follows from Theorem~\ref{thm1b}. If $\mn\cap\cB$ does not contain $d$ linearly independent vectors, then \eqref{eq4b} follows from Theorem~\ref{thm3b}. The proof of Corollary~\ref{cor1b} is complete.
\end{proof}

\begin{proof}[Proof of Corollary~\ref{cor1}]
In the light of Theorem~\ref{thm1}, the bound \eqref{eq2} follows from \eqref{eq2b}, while the bound \eqref{eq3} follows from \eqref{eq3b}. The proof of Corollary~\ref{cor1} is complete.
\end{proof}

\begin{proof}[Proof of Corollary~\ref{cor2}]
Assume that $\cB$ does not contain a lattice basis of $\mn$. Then, by an observation of Mahler~\cite{M2} (see also \cite[Ch.~2, \S 10.2]{GL}), the scaled body $\tfrac{1}{d}\cB$ does not contain $d$ linearly independent lattice vectors from $\mn$. Hence, by Corollary~\ref{cor1}, it follows that
\[\vol(\cB)\ll_d\vol(\tfrac{1}{d}\cB)\ll_d|\Delta|[\mo:\mn].\]
The proof of Corollary~\ref{cor2} is complete.
\end{proof}

\begin{proof}[Proof of Theorem~\ref{thm4b}] We borrow several ideas from the proof of Theorem~\ref{thm3b} without further mention. Let $x_1,\dotsc,x_m\in\mn$ be linearly independent lattice vectors whose Euclidean norms in $k\otimes_\QQ\RR$ are the successive minima $\lambda_1,\dotsc,\lambda_m$, respectively. Let $X$ be the $m$-set $\{x_1,\dotsc,x_m\}\subset\mn$, and let $V$ be the $\RR$-span of $X$. Then
$V$ is an $m$-dimensional $\RR$-subspace of $k\otimes_\QQ\RR$, and $\mn\cap V$ is an $m$-dimensional lattice in $V$ of successive minima $\lambda_1\leq\dotsb\leq\lambda_m$. In particular, the covolume of $\mn\cap V$ is $\asymp_d\lambda_1\cdots\lambda_m$. We fix an $m$-subset $S\subset\Sigma$ such that $\det(\sigma(x))^{\sigma\in S}_{x\in X}\neq 0$. For any $g\in G$, there exists an orthogonal projection $f$ of $k\otimes_\QQ\RR$ onto an $m$-subspace such that the covolume of $f(\mn\cap V)$ is at least $2^{-m}\left|\det(\sigma(x))^{\sigma\in gS}_{x\in X}\right|$. Since the covolume of $f(\mn\cap V)$ cannot exceed the covolume of $\mn\cap V$, we infer that
\[\lambda_1\cdots\lambda_m\gg_d\left|\det(\sigma(x))^{\sigma\in gS}_{x\in X}\right|,\qquad g\in G.\]
Taking the geometric mean of both sides over $g\in G$, and using \eqref{eq16}, we obtain \eqref{eq21}. Taking the reciprocal of \eqref{eq21}, and then multiplying both sides by \eqref{eq28}, we arrive at \eqref{eq20}. If $G$ is $2$-homogeneous, then the exponent of $|\Delta|$ in \eqref{eq16} can be improved to $|G|\frac{m(m-1)}{2d(d-1)}$, and our argument yields the following variants of \eqref{eq21} and \eqref{eq20}:
\begin{align}\label{eq21b}
\lambda_1\cdots\lambda_m&\gg_d|\Delta|^\frac{m(m-1)}{2d(d-1)}[\mo:\mn]^{\frac{m}{d}};\\
\label{eq20b}
\lambda_{m+1}\lambda_{m+2}\cdots\lambda_d&\ll_d|\Delta|^{\frac{(d-m)(d+m-1)}{2d(d-1)}}[\mo:\mn]^{1-\frac{m}{d}}.
\end{align}
The proof of Theorem~\ref{thm4b} is complete.
\end{proof}

\begin{proof}[Proof of Corollary~\ref{cor3}] We observe that \eqref{eq21} and \eqref{eq21b} are also valid for $m=d$, while \eqref{eq20} and \eqref{eq20b} are also valid for $m=0$. Indeed, these special cases amount to \eqref{eq28}. Now, taking the $m$-th root of \eqref{eq21} and \eqref{eq21b} readily yields the lower bound of \eqref{eq26} and \eqref{eq27}. Similarly, taking the $(d-m)$-th root of \eqref{eq20} and \eqref{eq20b} readily yields the upper bound of \eqref{eq26} and \eqref{eq27} with $m+1$ in place of $m$. The proof of Corollary~\ref{cor3} is complete.
\end{proof}

\section{Connections to the work of McMullen~\cite{M3} and Bhargava et al.~\cite{B}}\label{sect5}

If the number field $k$ is totally real, then we can identify the $\RR$-algebra $k\otimes_\QQ\RR$ with the set of column vectors $(z_\sigma)\in\RR^\Sigma$. The multiplicative group $(\RR^\Sigma)^\times$ acts on $\RR^\Sigma$ by multiplication, hence so does its subgroup
\[A:=\biggl\{(a_\sigma)\in(0,\infty)^\Sigma:\prod_{\sigma\in\Sigma}a_\sigma=1\biggr\}.\]
Let us consider the induced action of $A$ on the space of lattices of $\RR^\Sigma$. Geometrically, the space of lattices can be described as $\GL(\RR^\Sigma)/\GL(\ZZ^\Sigma)$, and the induced action of $A$ is given by left multiplication by positive diagonal matrices of determinant $1$. In particular, this action is continuous and preserves the covolume. The group of totally positive units $\mo^\times_+$ is cocompact in $A$ (cf. Dirichlet's unit theorem) and stabilizes the lattice $\mo$, hence the orbit $A\mo$ is compact. By a striking result of McMullen~\cite[Th.~4.1]{M3}, the compactness of $A\mo$ implies the existence of $a\in A$ such that the successive minima of the lattice $a\mo$ are equal: $\mu_1=\dots=\mu_d$. As we shall explain in the next paragraph, this fact gives rise to a short alternative proof of Corollary~\ref{cor2} (when $k$ is totally real). We note in passing that Levin, Shapira, Weiss~\cite[Th.~1.1]{LSW} have extended McMullen's theorem to closed orbits of lattices; these orbits arise from direct sums of totally real number fields and their full rank additive subgroups \cite[Prop.~5.7]{SW}.

Let $\mu$ be the common value of $\mu_1=\dots=\mu_d$, and let $\cD$ be the closed Euclidean unit ball in $\RR^\Sigma$ centered at the origin. Then $a\mo\cap\mu\cD$ contains $d$ linearly independent vectors. Let $\mn\subset\mo$ be a nonzero ideal, and let $\cB\subset\RR^\Sigma$ be an orthotope of the form $\prod_{\sigma\in\Sigma}[-B_\sigma,B_\sigma]$. We claim that if $\cB$ does not contain a lattice basis of $\mn$, then
\begin{equation}\label{eq24}
\vol(\cB)\leq(2d\mu)^d|\Delta|^{1/2}[\mo:\mn].
\end{equation}
This is sufficient for the conclusion of Corollary~\ref{cor2}, since $\mu^d=\mu_1\cdots\mu_d\asymp_d|\Delta|^{1/2}$. Let us assume that \eqref{eq24} is false. Then $\vol(a\mu^{-1}d^{-1}\cB)>2^d|\Delta|^{1/2}[\mo:\mn]$, hence Theorem~\ref{thm1} guarantees the existence of a nonzero lattice point $x\in\mn\cap a\mu^{-1}d^{-1}\cB$. By our initial remarks, $x\mo\cap xa^{-1}\mu\cD$ contains $d$ linearly independent vectors, so by $\mn\mo\subset\mn$ and $\cB\cD\subset\cB$ it follows that $\mn\cap d^{-1}\cB$ also contains $d$ linearly independent vectors. Finally, by the earlier quoted observation of Mahler~\cite{M2} (see also \cite[Ch.~2, \S 10.2]{GL}), we conclude that $\cB$ contains a lattice basis of $\mn$.

Corollary~\ref{cor2} can also be connected to the work of Bhargava et al.~\cite{B} in multiple ways. Let $k$ be an arbitrary number field, and let $\lambda_1\leq\dotsb\leq\lambda_d$ be the successive minima of $\mo$ embedded as a lattice in $k\otimes_\QQ\RR$. Then \cite[Th.~1.6]{B} states that \begin{equation}\label{eq25}
\lambda_d\ll_d|\Delta|^{1/d}.
\end{equation}
We claim that \eqref{eq25} follows from Corollary~\ref{cor2}, while a weaker version of Corollary~\ref{cor2} follows from \eqref{eq25}. To justify the first claim, we set $B_\sigma:=\frac{1}{d+1}\lambda_d$ for all $\sigma\in\Sigma$ in \eqref{eq1}. Clearly, $\cB$ contains no lattice basis of $\mo$, hence $\vol(\cB)\ll_d|\Delta|$ by Corollary~\ref{cor2}, which is equivalent to \eqref{eq25} by \eqref{eq19}. To justify the second claim, we start from \eqref{eq25}. Let $\mn\subset\mo$ be a nonzero ideal, and let $\cB\subset k\otimes_\QQ\RR$ be a convex body of the form \eqref{eq1} not containing a lattice basis of $\mn$. As $\mo\cap\lambda_d\cD$ contains $d$ linearly independent vectors, we can proceed as in the previous paragraph but with $a\in A$ (resp. $\mu$) replaced by $1\in k$ (resp. $\lambda_d$). We deduce the following variant of \eqref{eq25}:
\[\vol(\cB)\leq(2d\lambda_d)^d|\Delta|^{1/2}[\mo:\mn]\ll_d|\Delta|^{3/2}[\mo:\mn].\]
That is, \eqref{eq25} alone implies a version of Corollary~\ref{cor2} in which $|\Delta|$ is replaced by $|\Delta|^{3/2}$.


\bigskip
\bigskip

\begin{thebibliography}{10}

\bibitem{B} M. Bhargava, A. Shankar, T. Taniguchi, F. Thorne, J. Tsimerman, Y. Zhao, \emph{Bounds on $2$-torsion in class groups of number fields and integral points on elliptic curves}, {\tt arXiv:1701.02458}, J. Amer. Math. Soc., to appear

\bibitem{B2} H. F. Blichfeldt, \emph{Notes on geometry of numbers}, In: \emph{The October meeting of the San Francisco Section}, Bull. Amer. Math. Soc. \textbf{27} (1921), 149--153.

\bibitem{C2} P. J. Cameron, \emph{Finite permutation groups and finite simple groups}, Bull. London Math. Soc. \textbf{13} (1981), 1--22.

\bibitem{C} J. G. van der Corput, \emph{Verallgemeinerung einer Mordellschen Beweismethode in der Geometrie der Zahlen}, Acta Arith. \textbf{1} (1935), 62--66.; \emph{Zweite Mitteilung}, ibid. \textbf{2} (1936), 145--146.

\bibitem{DRS} J. A. De Loera, J. Rambau, F. Santos, \emph{Triangulations: Structures for algorithms and applications}, Algorithms and Computation in Mathematics, Vol. 25, Springer-Verlag, Berlin, 2010.

\bibitem{GL} P. M. Gruber, C. G. Lekkerkerker, \emph{Geometry of numbers}, 2nd edition, North-Holland Mathematical Library, Vol. 37, North-Holland Publishing Co., Amsterdam, 1987.

\bibitem{H} M. Huber, \emph{The classification of flag-transitive Steiner 3-designs}, Adv. Geom. \textbf{5} (2005), 195--221.

\bibitem{K} W. M. Kantor, \emph{Automorphism groups of designs}, Math. Z. \textbf{109} (1969), 246--252.

\bibitem{LSW} M. Levin, U. Shapira, B. Weiss, \emph{Closed orbits for the diagonal group and well-rounded lattices}, Groups Geom. Dyn. \textbf{10} (2016), 1211--1225.

\bibitem{M2} K. Mahler, \emph{A theorem on inhomogeneous diophantine inequalities}, Proc. Kon. Ned. Akad. Wet. \textbf{41} (1938), 634--637.

\bibitem{M3} C. T. McMullen, \emph{Minkowski's conjecture, well-rounded lattices and topological dimension}, J. Amer. Math. Soc. \textbf{18} (2005), 711--734.

\bibitem{M} H. Minkowski, \emph{\"Uber die positiven quadratischen Formen und \"uber kettenbruch\"ahnliche Algorithmen}, J. Reine Angew. Math. \textbf{107} (1891), 278--297.

\bibitem{N} J. Neukirch, \emph{Algebraische Zahlentheorie}, Springer-Verlag, Berlin, 1992.

\bibitem{RS} B. L. Rothschild, E. G. Straus, \emph{On triangulations of the convex hull of $n$ points}, Combinatorica \textbf{5} (1985), 167--179.

\bibitem{SW} U. Shapira, B. Weiss, \emph{On the Mordell-Gruber spectrum}, Int. Math. Res. Not. IMRN \textbf{2015}, \emph{no.} 14, 5518--5559.

\bibitem{T} T. Tao, \emph{An uncertainty principle for cyclic groups of prime order}, Math. Res. Lett. \textbf{12} (2005), 121--127.

\end{thebibliography}
\end{document}